\documentclass[11pt]{article}
\usepackage{amssymb,amsmath,amsthm,rotating,setspace,indentfirst,graphicx,float,mathtools,cases}
\usepackage{tikz}

\numberwithin{equation}{section}

\theoremstyle{plain}

\newtheorem{thm}{Theorem}[section]
\newtheorem{lem}{Lemma}[section]
\newtheorem{cor}{Corollary}[section]

\newtheorem{altthm}{Theorem}

\theoremstyle{definition}
\newtheorem{rem}{Remark}[section]

\begin{document}

\title{Existence of Nonnegative Solutions of Nonlinear Fractional
  Parabolic Inequalities}

\author{Steven D. Taliaferro\\
Mathematics Department\\
Texas A\&M University\\
College Station, TX 77843}

\date{}
\maketitle	

\begin{abstract}
We study the existence of nontrivial nonlocal nonnegative solutions
$u(x,t)$ of the nonlinear initial value problems
\begin{numcases}{}
(\partial_t -\Delta)^\alpha u\geq u^\lambda 
&in $\mathbb{R}^n \times\mathbb{R},\,n\geq1$ \notag\\
 u=0 &in $\mathbb{R}^n \times(-\infty,0)$ \notag
\end{numcases}
and

\begin{numcases}{}
C_1 u^\lambda \leq(\partial_t -\Delta)^\alpha u\leq C_2 u^\lambda 
&in $\mathbb{R}^n \times\mathbb{R},\,n\geq1$ \notag\\
 u=0 &in $\mathbb{R}^n \times(-\infty,0)$, \notag
\end{numcases}
where $\lambda,\alpha,C_1$, and $C_2$ are positive constants with
$C_1 <C_2$. We use the definition of the fractional heat operator
$(\partial_t -\Delta)^\alpha$ given in \cite{T} and compare our
results in the classical case $\alpha=1$ to known results.
\medskip


\end{abstract}

\section{Introduction}\label{introduction}
In this paper we study the existence of nontrivial nonlocal
nonnegative solutions $u(x,t)$ of the nonlinear initial value problems

\begin{numcases}{}
(\partial_t -\Delta)^\alpha u\geq u^\lambda 
&in $\mathbb{R}^n \times\mathbb{R},\,n\geq1$ \label{sup1}\\
 u=0 &in $\mathbb{R}^n \times(-\infty,0)$ \label{sup2}
\end{numcases}
and

\begin{numcases}{}
C_1 u^\lambda \leq(\partial_t -\Delta)^\alpha u\leq C_2 u^\lambda 
&in $\mathbb{R}^n \times\mathbb{R},\,n\geq1$ \label{ap1}\\
 u=0 &in $\mathbb{R}^n \times(-\infty,0)$, \label{ap2}
\end{numcases}
where $\lambda,\alpha,C_1$, and $C_2$ are positive constants with
$C_1 <C_2$. 

For a discussion of where the nonlocal fractional heat operator
$(\partial_t-\Delta)^\alpha$, $\alpha>0$, arises naturally in
applications, please see \cite{NS,ST}.

For each of the problems \eqref{sup1}, \eqref{sup2} and \eqref{ap1},
\eqref{ap2}, we compare our results in the classical case $\alpha=1$
to known results. Specifically, our result Theorem \ref{thm1.1} in
Section \ref{sec1.1} for the problem \eqref{sup1}, \eqref{sup2}
implies
\begin{enumerate}
\item[(i)] when $\lambda>1$ the existence of a critical exponent
  $\lambda_0(n,\alpha)>1$ for the nonexistence of nontrivial
  nonnegative solutions of the problem \eqref{sup1}, \eqref{sup2},
  which agrees with the well-known Fujita exponent $\lambda_F=1+2/n$
  when $\alpha=1$ (see Remark \ref{F}) and
\item[(ii)] when $0<\lambda<1$ a nonexistence result for nontrivial
  nonnegative solutions of the problem \eqref{sup1}, \eqref{sup2}
  which when $\alpha=1$ is similar to a result in \cite{AE} (see
  Section \ref{sec1.1}).
 \end{enumerate}

Similarly, our result Theorem \ref{thm1.2} in Section \ref{sec1.2} for
the problem \eqref{ap1}, \eqref{ap2} implies when $\lambda>1$ the
existence of a critical exponent $\lambda_1(n,\alpha)>1$ for the
existence of nontrivial nonnegative solutions of the problem \eqref{ap1},
\eqref{ap2} which agrees when $\alpha=1$ with a critical exponent in
\cite{HW} for the existence of self-similar solutions of the problem

\begin{numcases}{}
(\partial_t -\Delta)^\alpha u=u^\lambda
&in $\mathbb{R}^n \times\mathbb{R},\,n\geq1$ \label{ex1}\\
 u=0 &in $\mathbb{R}^n \times(-\infty,0)$ \label{ex2}
\end{numcases}
when $\alpha=1$ (see Section \ref{sec1.2}).

In order to complement our results for the two problems \eqref{sup1},
\eqref{sup2} and \eqref{ap1}, \eqref{ap2}, we recall in Section
\ref{sec1.3} our result in \cite{T} dealing with the existence of
nontrivial nonlocal nonnegative solutions of the initial value problem
\begin{numcases}{}
0\leq(\partial_t -\Delta)^\alpha u\leq u^\lambda
&in $\mathbb{R}^n \times\mathbb{R},\,n\geq1$ \label{sub1}\\
 u=0 &in $\mathbb{R}^n \times(-\infty,0)$. \label{sub2}
\end{numcases}

We refer to the four problems (\ref{sup1}, \ref{sup2}), (\ref{ap1},
\ref{ap2}), (\ref{ex1}, \ref{ex2}), and (\ref{sub1}, \ref{sub2}) as the
super problem, approximate problem, exact problem, and sub problem,
respectively.

As in \cite{T}, we define the fully fractional nonlocal heat operator
\begin{equation}\label{Ha}
 (\partial_t -\Delta)^\alpha :Y^{p}_{\alpha}\to X^p
\end{equation}
for
\begin{equation}\label{ap}
 \left(p>1\text{ and }0<\alpha<\frac{n+2}{2p}\right)\quad\text{ or }
\quad\left(p=1\text{ and }0<\alpha\leq\frac{n+2}{2p}\right)
\end{equation}as the inverse of the operator
\begin{equation}\label{I.5}
 J_\alpha :X^p \to Y^{p}_{\alpha}
\end{equation}
where
\begin{equation}\label{I.6}
 X^p:=\bigcap_{T\in\mathbb{R}}L^p (\mathbb{R}^n \times\mathbb{R}_T ),
\quad\mathbb{R}_T :=(-\infty,T),
\end{equation}
\begin{equation}\label{I.7}
 J_\alpha f(x,t):=\iint_{\mathbb{R}^n \times\mathbb{R}_t}
\Phi_\alpha (x-\xi,t-\tau)f(\xi,\tau)\,d\xi \,d\tau
\end{equation}
and
\begin{equation}\label{I.8}
 Y^{p}_{\alpha}:=J_\alpha (X^p ).
\end{equation}
By \eqref{I.6} we mean $X^p$ is the set of all measurable functions
$f:\mathbb{R}^n \times\mathbb{R}\to\mathbb{R}$ such that
$$\| f\|_{L^p (\mathbb{R}^n \times\mathbb{R}_T )}<\infty
\quad\text{ for all }T\in\mathbb{R}.$$
In the definition \eqref{I.7} of $J_\alpha$,
\begin{equation}\label{I.9}
\Phi_\alpha (x,t):=\frac{t^{\alpha-1}}{\Gamma(\alpha)}\,\frac{1}{(4\pi
  t)^{n/2}}e^{-|x|^2 /(4t)}\raisebox{2pt}{$\chi$}_{(0,\infty)}(t)
\end{equation}
is the fractional heat kernel.

When $p$ and $\alpha$ satisfy \eqref{ap}, it was shown in \cite{T}
that the operator \eqref{I.5} has the following properties:
\begin{enumerate}
\item[(P1)] it makes sense because
  $J_\alpha f\in L^{p}_{\text{loc}}(\mathbb{R}^n
  \times\mathbb{R})\text{ for }f\in X^p$,
\item[(P2)] it is one-to-one and onto, and
\item[(P3)] if $f\in X^p$ and $u=J_\alpha f$ then $f=0$ in
  $\mathbb{R}^n \times(-\infty,0)$ if and only if $u=0$ in
  $\mathbb{R}^n \times(-\infty,0)$.
\end{enumerate}
By properties (P1) and (P2) we can indeed define \eqref{Ha} as the
inverse of \eqref{I.5} when $p$ and $\alpha$ satisfy \eqref{ap}.
Property (P3) will be needed to handle the initial condition
$u=0$ in the above initial value problems.

Motivation for the above definition of \eqref{Ha} along with some more
of its properties can be found in \cite{T}.

Stinga and Torrea \cite{ST} (see also Nystr\"om and Sande \cite{NS})
gave an alternate definition of the fractional nonlocal heat operator
\[
(\partial_t-\Delta)^\alpha:U\to V,
\]
which agrees with our definition \eqref{Ha} on the intersection $U\cap
Y^p_\alpha $ of their domains. Functions
$u:\mathbb{R}^n\times\mathbb{R}\to \mathbb{R}$ in $U$ are required to
be bounded and sufficiently smooth. Their definition, unlike ours, is
well suited for studying the Dirichlet problem for 
\begin{equation}\label{I10}
(\partial_t-\Delta)^\alpha u=f \quad\text{in } \Omega\times(0,T)
\end{equation}
where $\Omega\subset\mathbb{R}^n$ is a bounded domain. However our
definition seems more suited for studying \eqref{I10} when
$\Omega=\mathbb{R}^n$ and $T=\infty$ because functions in $Y^p_\alpha$
can be discontinuous and locally unbounded, which allows for a greater
variety of solutions of \eqref{I10}.

The operator  \eqref{Ha} is a fully fractional heat operator as
opposed to time fractional heat operators in which the fractional
derivatives are only with respect to $t$, and space fractional heat
operators, in which the fractional derivatives are only with respect to $x$.

Some recent results for nonlinear PDEs containing time (resp. space)
fractional heat operators can be found in \cite{AV,A,ACV,DVV,K,
  KSVZ,M,OD,SS,VZ,ZS}
(resp. \cite{AABP,AMPP,BV,CVW,DS,FKRT,GW,JS,MT,PV,Su,V,VV,VPQR}). Except
for \cite{T}, we know of no results for nonlinear PDEs containing the
fully fractional heat operator $(\partial_t -\Delta)^\alpha$. However
results for linear PDEs containing this operator, including in
particular
\[
(\partial_t-\Delta)^\alpha u=f,
\]
where $f$ is a given function, can be found in \cite{ACM,NS,SK,ST}.

\section{Statement and relevance of results}\label{sec2}
In this section we state our results and relate them to results in
\cite{AE,F,HW,T}.  In order to do this, we first note that for each
fixed $p\geq1$ the open first quadrant of the $\lambda\alpha$-plane is
the union of the following pairwise disjoint sets which are graphed in
Figure \ref{fig1}:
\begin{align*}
 & A:=\{(\lambda,\alpha):0<\lambda<1\text{ and }\alpha>0\},\\
 & B:=\{(\lambda,\alpha):\lambda=1\text{ and }\alpha>0\},\\
 & C:=\left\{(\lambda,\alpha):\lambda>1\text{ and }\alpha\geq\frac{n+2}{2}\left(1-\frac{1}{\lambda}\right)\right\},\\
 & D:=\left\{(\lambda,\alpha):\lambda>1\text{ and }\frac{n+2}{2p}\left(1-\frac{1}{\lambda}\right)\leq\alpha<\frac{n+2}{2}\left(1-\frac{1}{\lambda}\right)\right\},\\
 & E:=\left\{(\lambda,\alpha):\lambda>1\text{ and }0<\alpha<\frac{n+2}{2p}\left(1-\frac{1}{\lambda}\right)\right\}.
\end{align*}

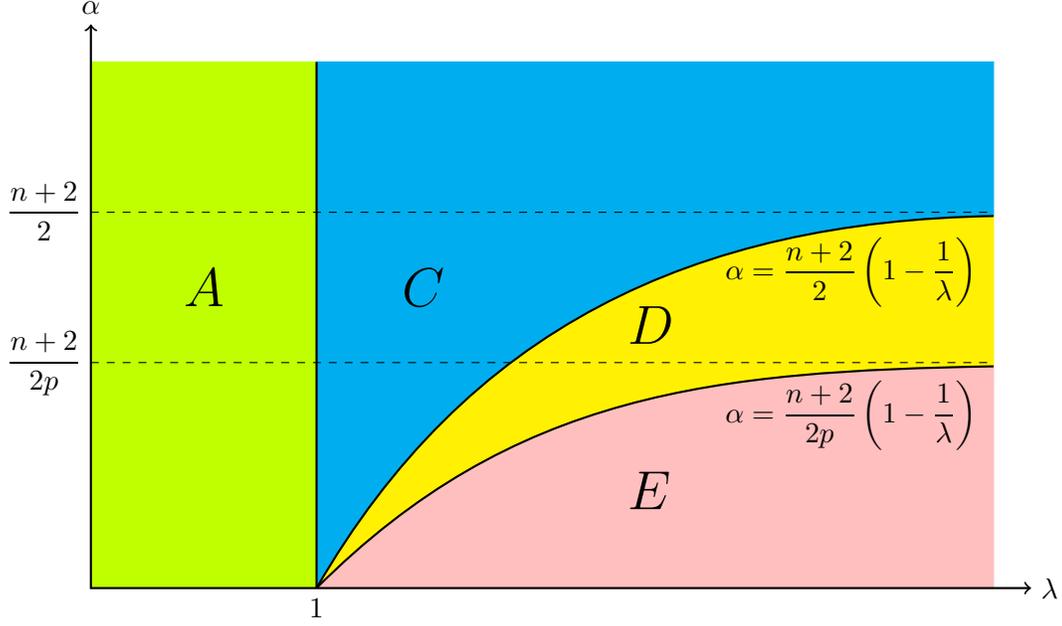
\begin{figure} 

\begin{tikzpicture}
\draw [fill=lime, lime] (0,0) rectangle (3,7);
\draw [fill=pink, pink] (3,0) to [out=45, in=181] (12,2.95) -- (12,0)
-- (3,0);
\draw [fill=cyan, cyan] (3,0) to [out=60, in=181] (12,4.95) -- (12,7)
-- (3,7) -- (3,0);
\draw [fill=yellow, yellow] (3,0) to [out=60, in=181] (12,4.95) -- (12,2.95)
to [out=181, in=45] (3,0);
\draw [<->] [thick] (0,7.5) -- (0,0) -- (12.5,0);
\draw [thick] (3,0) -- (3,7);
\node [left] at (0,2.99) {$\hskip 0.5in \dfrac{n+2}{2p}$};
\node [left] at (0,5.02) {$\hskip 0.5in \dfrac{n+2}{2}$};
\node [below] at (3,0) {1};
\draw [dashed] (0,3) -- (12,3);
\draw [dashed] (0,5) -- (12,5);
\draw [thick] (3,0) to [out=45, in=181] (12,2.95);
\draw [thick] (3,0) to [out=60, in=181] (12,4.95);
\node [right] at (8.3,2.3) {$\alpha=\dfrac{n+2}{2p}\left(1-\dfrac{1}{\lambda}\right)$};
\node [right] at (8.3,4.2) {$\alpha=\dfrac{n+2}{2}\left(1-\dfrac{1}{\lambda}\right)$};
\node [right] at (12.5,0) {$\lambda$};
\node [above] at (0,7.5) {$\alpha$};
\node [right] at (1.1,4) {\huge $A$};
\node [right] at (4.0,4) {\huge $C$};
\node [right] at (7,3.5) {\huge $D$};
\node [right] at (7,1.3) {\huge $E$};
\end{tikzpicture}

\caption{Graphs of the sets $A$, $B$, $C$, $D$, and $E$.}
\label{fig1}

\end{figure}

Note that if $p=1$ then $D=\emptyset$.

\subsection{The super problem}\label{sec1.1}
Our result for the super problem \eqref{sup1}, \eqref{sup2} is the
following. 

\begin{thm}\label{thm1.1}
  Suppose $\alpha$ and $p$ satisfy \eqref{ap} and $\lambda>0$.  Then
  the super problem \eqref{sup1}, \eqref{sup2} has a nontrivial
  nonnegative solution $u\in Y^{p}_{\alpha}$ if and only if
 $$(\lambda,\alpha)\in B\cup D\cup E.$$
\end{thm}

An immediate consequence of Theorem \ref{thm1.1} is the following
corollary.

\begin{cor}\label{cor1.1}
  Suppose $\alpha$ and $p$ satisfy \eqref{ap} and $0<\lambda<1$. Then
  the only nonnegative solution $u\in Y^{p}_{\alpha}$ of the super
  problem \eqref{sup1}, \eqref{sup2} is the trivial solution
  $u\equiv0$.
\end{cor}

A result similar to Corollary \ref{cor1.1} when $\alpha=1$ was proved
in \cite{AE} for mild nonnegative super solutions of the initial value
problem
\begin{align*}
 & (\partial_t -\Delta)u=u^\lambda \qquad\text{in }\mathbb{R}^n \times(0,\infty)\\
 & u(x,0)=u_0 (x) \qquad\text{for }x\in\mathbb{R}^n,
\end{align*}
where $0<\lambda<1$.

Since \eqref{ap} holds when $p=1$ and $0<\alpha\le (n+2)/2$ and since
$(\lambda,\alpha)\in D\cup E$ if and only if 
\[
0<\alpha<\frac{n+2}{2}\quad\text{and}\quad
\lambda>\lambda_0(n,\alpha):=1+\frac{2\alpha}{n+2-2\alpha}
 \]
(see Figure \ref{fig1}), we obtain
also from Theorem \ref{thm1.1} the following result.

\begin{cor}\label{cor1.2}
Suppose $0<\alpha\le (n+2)/2$ and $\lambda>1$. Then the super problem
\eqref{sup1}, \eqref{sup2} has a nontrivial nonnegative solution $u\in
Y^1_{\alpha}$  if and only if
\[
\alpha\not=(n+2)/2\quad\text{and}\quad 
\lambda>\lambda_0(n,\alpha).
\]
\end{cor}

For comparison with Corollary \ref{cor1.2}, we recall the famous
Fujita result \cite{F}, the following improved version of which
appears in \cite[Theorem 18.1]{QS}.

\begin{altthm}\label{thmA}
 If $1<\lambda\leq1+2/n$ then the only nonnegative solution $u\in L^{\lambda}_{\text{loc}}(\mathbb{R}^n \times(0,\infty))$ of the inequality
 $$(\partial_t -\Delta)u\geq u^\lambda \quad\text{in }\mathcal{D}^\prime (\mathbb{R}^n \times(0,\infty))$$
 is the trivial solution $u\equiv 0$.
\end{altthm}

\begin{rem}\label{F}
Since $\lambda_0(n,1)=1+2/n$, we see that Corollary \ref{cor1.2} can
be viewed as a fractional nonlocal version of Theorem \ref{thmA} and
$\lambda_0(n,\alpha)$ for $0<\alpha<(n+2)/2$ can be viewed as the
critical exponent of Fujita type for nonnegative solutions
$u\in Y^1_{\alpha}$ of the super problem \eqref{sup1}, \eqref{sup2}.
\end{rem}

For the proof of Theorem \ref{thm1.1} we will need the following lemma
of independent interest which gives in particular conditions for the
nonexistence of nontrivial nonnegative solutions of
$$(\partial_t -\Delta)^m u\geq u^\lambda \quad\text{in }
\mathcal{D}^\prime (\mathbb{R}^n \times(0,\infty))$$
where $m$ is a positive integer.

\begin{lem}\label{lem3.3}
 Suppose $m$ is a positive integer,
 \begin{equation}\label{L3.4}
  K>0,\quad (\lambda,\alpha)\in C,\quad\alpha\leq m,
 \end{equation}
 and $u\in L^{\lambda}_{\text{loc}}(\mathbb{R}^n \times(0,\infty))$ is
 a nonnegative solution of
 \begin{equation}\label{L3.2}
  (\partial_t -\Delta)^m u\geq(K(t+1)^{-(m-\alpha)}u)^\lambda 
\quad\text{in }\mathcal{D}^\prime (\mathbb{R}^n \times(0,\infty))
 \end{equation}
 such that
 \begin{equation}\label{L3.1}
  (\partial_t -\Delta)^j u\in L^1_{\text{loc}}(\mathbb{R}^n \times(0,\infty)),\qquad j=1,2,...,m-1
 \end{equation}
 and
 \begin{equation}\label{L3.3}
  (\partial_t -\Delta)^j u\geq0,\qquad j=1,2,...,m-1.
 \end{equation}
 Then
 \begin{equation}\label{L3.6}
  u=0\quad\text{in }\mathcal{D}^\prime (\mathbb{R}^n \times(0,\infty)).
 \end{equation}
\end{lem}

\begin{rem}\label{1.2}
  Since $(\lambda,1)\in C$ if and only if $1<\lambda\le 1+2/n$, we see
  that a consequence of Lemma \ref{lem3.3} with $m=\alpha=K=1$ is
  Theorem \ref{thmA}.
\end{rem}

\subsection{The approximate and exact problems}\label{sec1.2}

If $u\in Y^p_\alpha$, where $\alpha$ and $p$ satisfy \eqref{ap}, is a
solution of the exact problem \eqref{ex1}, \eqref{ex2} then for all
$\beta>0$ so is 
\begin{equation}\label{ss1}
u_\beta(x,t):=\beta^{-\frac{2\alpha}{\lambda-1}}u(x/\beta,t/\beta^2).
\end{equation}
If, in addition, $u$ is self-similar, that is $u_\beta=u$ for all
$\beta>0$, then substituting $\beta=\sqrt{t}$, $t>0$, in \eqref{ss1},
we see that 
\begin{equation}\label{ss2}
u(x,t)=\begin{cases}
t^{-\frac{\alpha}{\lambda-1}}u(x/\sqrt{t},1)
&\text{for }(x,t)\in\mathbb{R}^n\times(0,\infty)\\
0&\text{for }(x,t)\in\mathbb{R}^n\times(-\infty,0].\\
\end{cases}
\end{equation}
Moreover, any function $u$ satisfying \eqref{ss2} is
self-similar. Inspired by \cite{HW}, we will seek in this section
solutions of the approximate problem \eqref{ap1}, \eqref{ap2} of the
form \eqref{ss2}.

We have no results for solutions $u\in Y^{p}_{\alpha}$ of the
approximate problem \eqref{ap1}, \eqref{ap2} when $p>1$ and
$(\lambda,\alpha)$ lies in the curve
\begin{equation}\label{curve}
 \alpha=\frac{n+2}{2p}\left(1-\frac{1}{\lambda}\right),\quad 1<\lambda<\infty,
\end{equation}
which is graphed in Figure \ref{fig1}.
Otherwise we have the following result.

\begin{thm}\label{thm1.2}
 Suppose $\lambda>0$ and either
 \begin{enumerate}
  \item[(i)] $\alpha$ and $p$ satisfy \eqref{ap}$_2$, or
  \item[(ii)] $\alpha$ and $p$ satisfy \eqref{ap}$_1$ and the point $(\lambda,\alpha)$ does not lie on the curve \eqref{curve}.
 \end{enumerate}
 Then there exist positive constants $C_1$ and $C_2$ such that the
 approximate problem \eqref{ap1}, \eqref{ap2} has a nontrivial
 nonnegative solution $u\in Y^{p}_{\alpha}$ if and only if
 $$(\lambda,\alpha)\in E.$$
In this case, such a solution is given by
\begin{equation}\label{ss3}
u(x,t)=\begin{cases}
t^{-\frac{\alpha}{\lambda-1}}w_\alpha(x/\sqrt{t})
&\text{for }(x,t)\in\mathbb{R}^n\times(0,\infty)\\
0&\text{for }(x,t)\in\mathbb{R}^n\times(-\infty,0]\\
\end{cases}
\end{equation}
where
\begin{equation}\label{ss4}
w_\alpha(z)=e^{-\frac{|z|^2}{4}}
(|z|^2
+1)^{-(\frac{n+2}{2}-\frac{\alpha\lambda}{\lambda-1})}
\quad\text{for }z\in\mathbb{R}^n.
\end{equation}
\end{thm}

An immediate consequence of Theorem \ref{thm1.2} is the following corollary.

\begin{cor}\label{cor1.3}
 Suppose
 \begin{equation}\label{apl2}
  p\geq1,\quad 0<\alpha<\frac{n+2}{2p},\quad\text{ and }\quad\lambda>\frac{n+2}{n+2-2p\alpha}.
 \end{equation}
 Then there exist positive constants $C_1$ and $C_2$ such that a
 nontrivial nonnegative solution $u\in Y^{p}_{\alpha}$ of the
 approximate problem \eqref{ap1}, \eqref{ap2} is given by \eqref{ss3}
 where $w_\alpha$ is defined in \eqref{ss4}.
\end{cor}

Since the conditions \eqref{apl2} hold if
\[
\alpha=1,\quad 1\leq p<\frac{n+2}{2},\quad\text{and}\quad\lambda>\frac{n+2}{n+2-2p},
\]
we see that Corollary \ref{cor1.3} can be viewed as an $\alpha\neq1$
version of the following $\alpha=1$ result in \cite{HW} for the exact
problem \eqref{ex1}, \eqref{ex2}.

\begin{altthm}[\cite{HW}]\label{thmB}
 Suppose
 \begin{equation*}
  1\leq p<\frac{n+2}{2}\quad\text{and}\quad\frac{n+2}{n+2-2p}<\lambda<
 \begin{cases}
  \frac{n+2}{n-2} & \text{if }n\geq3\\
  \infty & \text{if }n=1\text{ or }2.
 \end{cases}
 \end{equation*}
 Then a nontrivial nonnegative solution $u\in Y^p_1$ of
 the exact problem \eqref{ex1}, \eqref{ex2} with $\alpha=1$ is given by
\begin{equation}\label{ss5}
u(x,t)=\begin{cases}
t^{-\frac{1}{\lambda-1}}w(x/\sqrt{t})
&\text{for }(x,t)\in\mathbb{R}^n\times(0,\infty)\\
0&\text{for }(x,t)\in\mathbb{R}^n\times(-\infty,0]\\
\end{cases}
\end{equation}
for some positive radial function $w:\mathbb{R}^n\to\mathbb{R}$ such
that 
\[
\frac{w(z)}{w_1(z)}\quad\text{is bounded between positive constants on
} \mathbb{R}^n
\]
where $w_\alpha$ is defined in \eqref{ss4}.
\end{altthm}

We have no results for the exact problem, but Corollary \ref{cor1.3}
and Theorem \ref{thmB} motivate the following open question.
\medskip

\noindent{\bf Open question}. Suppose $\alpha$ and $p$ satisfy \eqref{ap} so
that the operator \eqref{Ha} is defined. For what $\lambda>0$ does
there exist a nontrivial nonnegative solution $u\in Y^p_\alpha$ of the
exact problem \eqref{ex1}, \eqref{ex2}?
\medskip

Since any solution of the exact problem \eqref{ex1}, \eqref{ex2} is
also a solution of the approximate problem \eqref{ap1}, \eqref{ap2}, it
follows from Theorem \ref{thm1.2} that a necessary condition on
$\lambda$ is that either $p>1$ and $(\lambda,\alpha)$ lies on the
curve \eqref{curve} or $(\lambda,\alpha)\in E$ (see Figure \ref{fig1}).

\subsection{The sub problem}\label{sec1.3}
We have no results for solutions $u\in Y^{p}_{\alpha}$ of the sub
problem \eqref{sub1}, \eqref{sub2} when $p\geq 1$ and the point
$(\lambda,\alpha)$ lies on the curve \eqref{curve}.  Otherwise we have
the following result from \cite{T}.

\begin{thm}\label{thm1.3}
  Suppose $\alpha$ and $p$ satisfy \eqref{ap}, $\lambda>0$, and the point
  $(\lambda,\alpha)$ does not lie on the curve \eqref{curve}.  Then
  the sub problem \eqref{sub1}, \eqref{sub2} has a nontrivial
  nonnegative solution $u\in Y^{p}_{\alpha}$ if and only if
 \[
(\lambda,\alpha)\in A\cup E.
\]
\end{thm}

\section{$J_\alpha$ version of results}\label{sec3}
We define the $J_\alpha$ versions of the super, approximate, exact, and
sub problems in Section \ref{introduction} to be respectively the problems

\begin{numcases}{}
f\geq(J_\alpha f)^\lambda 
&in $\mathbb{R}^n \times\mathbb{R}$ \label{Jsup1}\\
 f=0 &in $\mathbb{R}^n \times(-\infty,0)$; \label{Jsup2}
\end{numcases}

\begin{numcases}{}
C_1 f\leq(J_\alpha f)^\lambda \leq C_2 f 
&in $\mathbb{R}^n \times\mathbb{R}$ \label{Jap1}\\
 f=0 &in $\mathbb{R}^n \times(-\infty,0)$; \label{Jap2}
\end{numcases}

\begin{numcases}{}
 f=(J_\alpha f)^\lambda 
&in $\mathbb{R}^n \times\mathbb{R}$ \label{Jex1}\\
 f=0 &in $\mathbb{R}^n \times(-\infty,0)$; \label{Jex2}
\end{numcases}
and

\begin{numcases}{}
0\leq f\leq(J_\alpha f)^\lambda 
&in $\mathbb{R}^n \times\mathbb{R}$ \label{Jsub1}\\
 f=0 &in $\mathbb{R}^n \times(-\infty,0)$; \label{Jsub2}
\end{numcases}
which we will refer to respectively as the super $J_\alpha$ problem,
approximate $J_\alpha$ problem, exact $J_\alpha$ problem, and
sub $J_\alpha$ problem (or collectively as the $J_\alpha$ problems).

If $\alpha$ and $p$ satisfy \eqref{ap} and $\lambda>0$ then by
properties (P1)--(P3) in Section \ref{introduction} of $ J_\alpha$ and the definition of
the fractional heat operator \eqref{Ha}, $u$ is a nonnegative solution
in $Y^{p}_{\alpha}$ of the super (approximate, exact, sub) problem in
Section \ref{introduction} if and only if
$$f:=(\partial_t -\Delta)^\alpha u$$
is a nonnegative solution in $X^p$ of the super (approximate, exact,
sub) $J_\alpha$ problem in this section. (However, the positive
constants $C_1$ and $C_2$ in \eqref{ap1}, \eqref{ap2} may be different
than the positive constants $C_1$ and $C_2$ in \eqref{Jap1},
\eqref{Jap2}.)

Since we will only consider solutions $f$ of the $J_\alpha$ problems
which are nonnegative on $\mathbb{R}^n \times\mathbb{R}$, $J_\alpha f$
in these problems will always be a well-defined nonnegative extended
real valued function on $\mathbb{R}^n \times\mathbb{R}$ even when the
condition \eqref{ap} is replace with the weaker condition that
\begin{equation}\label{apalt}
 p\in[1,\infty)\quad\text{and}\quad\alpha>0.
\end{equation}
Hence in this section we study the $J_\alpha$ problems with condition
\eqref{ap} replaced with \eqref{apalt}.  However our results in this
section for the $J_\alpha$ problems will only yield corresponding
results for the original versions of these problems in Section
\ref{introduction} when \eqref{ap} holds, for otherwise the fractional
heat operators in these original problems are not defined.  (For a
more detailed discussion of the properties of $J_\alpha$ when
\eqref{ap} does not hold see \cite[Section 4]{T}.)

Under the equivalence discussed above of the $J_\alpha$ problems and
original versions of these problems in Section \ref{introduction}, the
following Theorems \ref{thm2.1}--\ref{thm2.3}, when restricted to the
case that \eqref{ap} holds, clearly imply Theorem
\ref{thm1.1}--\ref{thm1.3} respectively.

\begin{thm}\label{thm2.1}
  Suppose $p$ and $\alpha$ satisfy \eqref{apalt} and $\lambda>0$.
  Then the super $J_\alpha$ problem \eqref{Jsup1}, \eqref{Jsup2} has a
  nontrivial nonnegative solution $f\in X^p$ if and only if
 $$(\lambda,\alpha)\in B\cup D\cup E.$$
\end{thm}

\begin{thm}\label{thm2.2}
  Suppose $\lambda,\alpha>0$ and either
 \begin{enumerate}
  \item[(i)] $p=1$ or
  \item[(ii)] $p\in(1,\infty)$ and the point $(\lambda,\alpha)$ does
    not lie of the curve \eqref{curve}.
 \end{enumerate}
 Then there exist positive constants $C_1$ and $C_2$ such that the
 approximate $J_\alpha$ problem \eqref{Jap1}, \eqref{Jap2} has a
 nontrivial nonnegative solution $f\in X^p$ if and only if
 \[
(\lambda,\alpha)\in E.
\]
In this case, such a solution is given by $f=u^\lambda$, where $u$ is
defined by \ref{ss3}.
\end{thm}

\begin{thm}\label{thm2.3}
  Suppose $\alpha$ and $p$ satisfy \eqref{apalt}, $\lambda>0$, and the
  point $(\lambda,\alpha)$ does not lie on the curve \eqref{curve}.
  Then the sub $J_\alpha$ problem \eqref{Jsub1}, \eqref{Jsub2} has a
  nontrivial nonnegative solution $f\in X^p$ if and only if
 $$(\lambda,\alpha)\in A\cup E.$$
\end{thm}

We will prove Theorems \ref{thm2.1} and \ref{thm2.2} in Section
\ref{sec5}.  We proved Theorem \ref{thm2.3} in \cite{T}.

\section{Preliminary results}\label{sec4}
In this section we provide some remarks and lemmas needed for the
proofs of our results in Section \ref{sec3} dealing with solutions of the super
$J_\alpha$ problem \eqref{Jsup1}, \eqref{Jsup2} and the approximate
$J_\alpha$ problem \eqref{Jap1}, \eqref{Jap2}.

\begin{lem}\label{lem3.1}
 Suppose $\alpha,\beta\in(0,\infty),\:x\in\mathbb{R}^n$, and $0<\tau<t$.  Then
 \begin{equation}\label{L1.1}
  \int_{\mathbb{R}^n}\Phi_\alpha (x-\xi,t-\tau)\Phi_\beta (\xi,\tau)\,d\xi=\frac{(t-\tau)^{\alpha-1}\tau^{\beta-1}}{\Gamma(\alpha)\Gamma(\beta)}\Phi_1 (x,t).
 \end{equation}
\end{lem}

\begin{proof}
 Denote the left side of \eqref{L1.1} by $h(x,t,\tau)$.  Using the convolution theorem and the well-known fact that the Fourier transform with respect to $x$ of $\Phi_\alpha (x,t)$ is given by 
 \begin{equation}\label{L1.2}
  \hat{\Phi}_\alpha
  (\cdot,t)(y)=\frac{t^{\alpha-1}}{\Gamma(\alpha)}e^{-t|y|^2}
\quad\text{for }t>0\text{ and }y\in\mathbb{R}^n,
 \end{equation}
 we find for $0<\tau<t$ that
 \begin{align*}
  \hat{h}(\cdot,t,\tau)(y) & =\left(\frac{(t-\tau)^{\alpha-1}}{\Gamma(\alpha)}e^{-(t-\tau)|y|^2}\right)\left(\frac{\tau^{\beta-1}}{\Gamma(\beta)}e^{-\tau|y|^2}\right)\\
  & =\frac{(t-\tau)^{\alpha-1}\tau^{\beta-1}}{\Gamma(\alpha)\Gamma(\beta)}e^{-t|y|^2}\\
  & =\frac{(t-\tau)^{\alpha-1}\tau^{\beta-1}}{\Gamma(\alpha)\Gamma(\beta)}\hat{\Phi}_1 (\cdot,t)(y)
 \end{align*}
 which proves \eqref{L1.1}.
\end{proof}

\begin{lem}\label{lem3.2}
 Suppose $\lambda,\alpha,T\in(0,\infty),\:p\in[1,\infty]$, and 
 \[ 
f(x,t)=g(x,t+T)\raisebox{2pt}{$\chi$}_{[0,\infty)}(t)\quad\text{for
}(x,t)\in\mathbb{R}^n \times\mathbb{R}
\]
 where $g:\mathbb{R}^n \times\mathbb{R}\to[0,\infty)$ is a measurable function such that
 \begin{equation}\label{L2.1}
  \| g\|_{L^p (\mathbb{R}^n \times(T,\hat{T}))}<\infty\quad\text{for all }\hat{T}>T
 \end{equation}
 and
 \begin{equation}\label{L2.2}
  g\geq(J_\alpha g)^\lambda \quad\text{in }\mathbb{R}^n \times\mathbb{R}.
 \end{equation}
 Then $f\in X^p$ and $f$ is a solution of the super $J_\alpha$ problem
 \eqref{Jsup1}, \eqref{Jsup2}.
\end{lem}

\begin{proof}
 For $t>0$ we have
 \begin{align*}
  \| f\|_{L^p (\mathbb{R}^n \times(-\infty,t))} & =\| f\|_{L^p (\mathbb{R}^n \times(0,t))}\\
  & =\| g\|_{L^p (\mathbb{R}^n \times(T,t+T))}<\infty
 \end{align*}
 by \eqref{L2.1}.  Thus $f\in X^p$.
 
 Clearly $f$ satisfies \eqref{Jsup2}.  Since \eqref{Jsup1} clearly
 holds in $\mathbb{R}^n \times(-\infty,0]$, it remains only to prove
 \eqref{Jsup1} holds in $\mathbb{R}^n \times(0,\infty)$.
 
 For $(x,t)\in\mathbb{R}^n \times(0,\infty)$ we find from \eqref{L2.2} that
 \begin{align*}
  (J_\alpha f)(x,t) & =\int^{t}_{0}\int_{\mathbb{R}^n}
\Phi_\alpha (x-\xi,t-\tau)g(\xi,\tau +T)\,d\xi\,d\tau\\
  & =\int^{\hat{t}}_{T}\int_{\mathbb{R}^n}\Phi_\alpha (x-\xi,\hat{t}-\hat{\tau})g(\xi,\hat{\tau})\,d\xi\,d\hat{\tau}\quad\text{where }\hat{t}=t+T\text{ and }\hat{\tau}=\tau+T\\
  & \leq\int^{\hat{t}}_{-\infty}\int_{\mathbb{R}^n}\Phi_\alpha (x-\xi,\hat{t}-\hat{\tau})g(\xi,\hat{\tau})\,d\xi\,d\hat{\tau}\\
  & =(J_\alpha g)(x,\hat{t})\\
  & \leq g(x,\hat{t})^{1/ \lambda}\\
  & =g(x,t+T)^{1/ \lambda}=f(x,t)^{1/ \lambda}.
 \end{align*}
 Thus \eqref{Jsup1} holds in $\mathbb{R}^n \times(0,\infty)$.
\end{proof}

\begin{rem}\label{rem3.1}
 Note for use in Lemma \ref{lem3.2} that if
 \[
0\leq g(x,t)\leq\psi(t)\Phi_\beta (x,t)\quad\text{for
}(x,t)\in\mathbb{R}^n \times\mathbb{R}
\]
 where $\beta\in(0,\infty)$ and $\psi:\mathbb{R}\to[0,\infty)$ is a continuous function then for $0<T<\hat{T}<\infty$ we have
 $$\| g\|_{L^\infty (\mathbb{R}^n \times(T,\hat{T}))}<\infty$$
 and for $p\in[1,\infty)$
 \begin{align*}
  & \| g\|^{p}_{ L^p (\mathbb{R}^n \times(T,\hat{T}))}\leq\int^{\hat{T}}_{T}\psi(t)^p \left(\frac{t^{\beta-1-n/2}}{\Gamma(\beta)(4\pi)^{n/2}}\right)^p \left(\int_{\mathbb{R}^n}e^{\frac{-p|x|^2}{4t}}dx\right)dt\\
  & \leq(\hat{T}-T)\max_{T\leq t\leq\hat{T}}\psi(t)^p \left(\frac{t^{\beta-1-n/2}}{\Gamma(\beta)(4\pi)^{n/2}}\right)^p \int_{\mathbb{R}^n}e^{\frac{-p|x|^2}{4\hat{T}}}dx<\infty .
 \end{align*}
 Thus $g$ satisfies \eqref{L2.1}.
\end{rem}

Our proof of Lemma \ref{lem3.3} is a modification of \cite[Proof of
Theorem 18.1(i)]{QS} and in particular requires the following lemma.
See \cite[pages 101-102]{QS} for its proof.

\begin{lem}\label{lem3.4}
 Suppose $v,f\in L^1_{\text{loc}}(\mathbb{R}^n \times(0,\infty))$ are nonnegative functions such that
 \[
(\partial_t -\Delta)v\geq f\quad\text{in }\mathcal{D}^\prime
(\mathbb{R}^n \times(0,\infty)).
\]
 Let
 \[\varphi\in C^{\infty}_{0}(B_1 (0))\quad\text{and}\quad\psi\in
 C^{\infty}_{0}((-1,1))
\]
 be nonnegative functions such that 
\[
\varphi=1 \text { in }B_{1/2}(0),\quad \psi=1 \text{ in } [0,1/2), 
\quad\text{and }\varphi,\psi\leq 1.  
\]
For
 $R>1$, $\beta>2$, and $t_0 >0$ define 
 \[\varphi_R (x)=\varphi\left(\frac{x}{R}\right)^\beta \quad\text{ for
 }x\in\mathbb{R}^n
\]
 and
 \[\psi_R (t)=\psi\left(\frac{t-t_0}{R^2}\right)^\beta \quad\text{for }t\geq t_0 .\]
 Then
 $$\int^{\infty}_{t_0}\int_{\mathbb{R}^n}f\varphi_R \psi_R \,dx\,dt\leq\frac{C}{R^2}\iint_{Q_R}v(\varphi_R \psi_R )^{\frac{\beta-2}{\beta}}\,dx\,dt$$
 where $C>0$ does not depend on $R$ and
 \begin{equation}\label{L4.1}
  Q_R =(B_R(0)\times(t_0 ,t_0 +R^2 ))\backslash (B_{R/2}(0)\times(t_0 ,t_0 +R^{2}/2)).
 \end{equation}
\end{lem}

\begin{proof}[Proof of Lemma \ref{lem3.3}]
  For $R>1,\:\gamma>2m$, and $t_0 >0$ define
 $$\varphi_R (x)=\varphi\left(\frac{x}{R}\right)^\gamma \quad\text{and}\quad\psi_R (t)=\psi\left(\frac{t-t_0}{R^2}\right)^\gamma$$
 for $x\in\mathbb{R}^n$ and $t\geq t_0$ where $\varphi$ and $\psi$ are
 as in Lemma \ref{lem3.4}.  Let $f_0 (x,t)$ be the function on the right side
 of \eqref{L3.2}.
 
 For $j=1,...,m$ we claim that
 \begin{equation}\label{L3.7}
  \int^{\infty}_{t_0}\int_{\mathbb{R}^n}f_0 \varphi_R \psi_R \,dx\,dt\leq\left(\frac{C}{R^2}\right)^j \iint_{Q_R}(H^{m-j}u)(\varphi_R \psi_R )^{\frac{\gamma-2j}{\gamma}}\,dx\,dt
 \end{equation}
 where $H=\partial_t-\Delta$, $C>0$ does not depend on $R$, and $Q_R$
 is defined in \eqref{L4.1}.
 
 Inequality \eqref{L3.7} holds for $j=1$ by \eqref{L3.2}--\eqref{L3.3}
 and Lemma \ref{lem3.4} with $v=H^{m-1}u$ and $f=f_0$.  Suppose inductively
 that \eqref{L3.7} is true for some integer $j\in[1,m-1]$.  Let
 \[\hat{\varphi}_R (x):=\varphi_R (x)^{\frac{\gamma-2j}{\gamma}}=\varphi\left(\frac{x}{R}\right)^{\gamma-2j}\quad\text{for }x\in\mathbb{R}^n\]
 and
 \[\hat{\psi}_R (t):=\psi_R (t)^{\frac{\gamma-2j}{\gamma}}=\psi\left(\frac{t-t_0}{R^2}\right)^{\gamma-2j}\quad\text{for }t\geq t_0.\]
 Then using the inductive assumption, \eqref{L3.1}, \eqref{L3.3}, and
 Lemma \ref{lem3.4} with 
\[
f=H^{m-j}u,\quad v=H^{m-j-1}u,\quad \text{and}\quad \beta=\gamma-2j,
\]
we find that
 \begin{align*}
  \int^{\infty}_{t_0}\int_{\mathbb{R}^n} & f_0 \varphi_R \psi_R \,dx\,dt\leq\left(\frac{C}{R^2}\right)^j \int^{\infty}_{t_0}\int_{\mathbb{R}^n}(H^{m-j}u)\hat{\varphi}_R \hat{\psi}_R \,dx\,dt\\
  & \leq\left(\frac{C}{R^2}\right)^j \left[\left(\frac{C}{R^2}\right)\iint_{Q_R}(H^{m-j-1}u)(\hat{\varphi}_R \hat{\psi}_R )^{\frac{\gamma-2j-2}{\gamma-2j}}\,dx\,dt\right]\\
  & =\left(\frac{C}{R^2}\right)^{j+1}\iint_{Q_R}(H^{m-(j+1)}u)(\varphi_R \psi_R )^{\frac{\gamma-2(j+1)}{\gamma}}\,dx\,dt
 \end{align*}
 which completes the inductive proof of \eqref{L3.7} for $j=1,2,...,m$.
 
 Taking $j=m$ in \eqref{L3.7}, defining 
$\gamma>2m$ and $\lambda^\prime >1$ by
 \[\frac{2m}{\gamma}=\frac{1}{\lambda^\prime}=1-\frac{1}{\lambda},\]
 and using H\"{o}lder's inequality, we have
 \begin{align*}
  \int^{\infty}_{t_0}\int_{\mathbb{R}^n} & ((t+1)^{-(m-\alpha)}u)^\lambda \varphi_R \psi_R \,dx\,dt\\
  & \leq\frac{C}{R^{2m}}\iint_{Q_R}u(\varphi_R \psi_R )^{\frac{\gamma-2m}{\gamma}=\frac{1}{\lambda}}\\
  & =\frac{C}{R^{2m}}\iint_{Q_R}(t+1)^{m-\alpha}(t+1)^{-(m-\alpha)}u(\varphi_R \psi_R )^{\frac{1}{\lambda}}\,dx\,dt\\
  & \leq\frac{C}{R^{2m}}I(R)\left(\iint_{Q_R}((t+1)^{-(m-\alpha)}u)^\lambda \varphi_R \psi_R \,dx\,dt\right)^{1/\lambda}
 \end{align*}
 where
 \begin{align*}
  I(R) & =\left(\int^{t_0 +R^2}_{t_0}\int_{|x|<R}(t+1)^{\lambda^\prime (m-\alpha)}\,dx\,dt\right)^{1/ \lambda^\prime}\\
  & \leq C(R^{n+2}R^{2\lambda^\prime (m-\alpha)})^{1/\lambda^\prime}\quad\text{by } \eqref{L3.4}\\
  & =C\: R^{\frac{n+2}{\lambda^\prime}+2(m-\alpha)}.
 \end{align*}
 Hence
 \begin{equation}\label{L3.8}
  \int^{\infty}_{t_0}\int_{\mathbb{R}^n}((t+1)^{-(m-\alpha)}u)^\lambda \varphi_R \psi_R \,dx\,dt\leq CR^{\frac{n+2}{\lambda^\prime}-2\alpha}\left(\iint_{Q_R}((t+1)^{-(m-\alpha)}u)^\lambda (\varphi_R \psi_R )\,dx\,dt\right)^{1/ \lambda}
 \end{equation}
 which implies
 \begin{equation}\label{L3.9}
  \int^{\infty}_{t_0}\int_{\mathbb{R}^n}((t+1)^{-(m-\alpha)}u)^\lambda \varphi_R \psi_R \,dx\,dt\leq CR^{(\frac{n+2}{\lambda^\prime}-2\alpha)\lambda^\prime}.
 \end{equation}
By \eqref{L3.4}, $2\alpha\ge (n+2)/\lambda^\prime$. Hence, sending
 $R$ to $\infty$ in \eqref{L3.9} we find that 
\[
\int^{\infty}_{t_0}\int_{\mathbb{R}^n}((t+1)^{-(m-\alpha)}u)^\lambda
\,dx\,dt<\infty,
\]
which implies the integral on the right side of \eqref{L3.8} tends to
zero as $R\to\infty$. Thus sending $R$ to $\infty$ in \eqref{L3.8}
yields $u=0$ in $\mathbb{R}^n \times (t_0,\infty)$. Hence, since
$t_0>0$ was arbitrary, we see that \eqref{L3.6} holds.
\end{proof}

\begin{lem}\label{new}
 Suppose $p\in[1,\infty)$ and $f\in X^p$ is a nonnegative function satisfying \eqref{Jsup2}.  Then
 \begin{equation}\label{N1}
  J_\alpha f\in L^1_{\text{loc}}(\mathbb{R}^n \times(0,\infty))\quad\text{for }\alpha>0,
 \end{equation}
 \begin{equation}\label{N2}
  HJ_\alpha f=J_{\alpha-1}f\quad\text{in }\mathcal{D}^\prime (\mathbb{R}^n \times(0,\infty))\text{ for }\alpha>1,
 \end{equation}
 and
 \begin{equation}\label{N3}
  HJ_1 f=f\quad\text{in }\mathcal{D}^\prime (\mathbb{R}^n \times(0,\infty)),
 \end{equation}
 where $H=\partial_t -\Delta$.
\end{lem}

\begin{proof}
 We will need the following easily-verified and/or well-known facts:
 \begin{enumerate}
  \item[(i)] 
   \begin{equation}\label{N4}
    \iint_{\mathbb{R}^n \times\mathbb{R}}\Phi_1 (x-\xi,t-\tau)H^* \varphi(x,t)dxdt=\varphi(\xi,\tau)
   \end{equation}
   for $\varphi\in C^{\infty}_{0}(\mathbb{R}^n \times\mathbb{R})$
   where 
$H^*= \partial_t +\Delta$; and
  \item[(ii)]
   \begin{equation}\label{N5}
    \int^{b}_{0}\int_{B_b (0)}\Phi_\alpha (x-\xi,t-\tau)dxdt\in L^q (\mathbb{R}^n \times(0,b))
   \end{equation}
   for $b,\alpha\in(0,\infty)$ and $q\in[1,\infty]$.
 \end{enumerate}

 To prove \eqref{N1}, let $b,\alpha\in(0,\infty)$.  Then by \eqref{N5} and H\"{o}lder's inequality we have
 \begin{align*}
  & \int^{b}_{0}\int_{B_b (0)}J_\alpha f\,dxdt\\
  & =\int^{b}_{0}\int_{B_b (0)}\int^{b}_{0}\int_{\mathbb{R}^n}\Phi_\alpha (x-\xi,t-\tau)f(\xi,\tau)d\xi d\tau\,dxdt\\
  & =\int^{b}_{0}\int_{\mathbb{R}^n}f(\xi,\tau)\left(\int^{b}_{0}\int_{B_b (0)}\Phi_\alpha (x-\xi,t-\tau)dxdt\right)d\xi d\tau<\infty.
 \end{align*}
 Thus \eqref{N1} holds.
 
 To prove \eqref{N2}, suppose $\beta,\gamma\in(0,\infty)$, $\beta+\gamma=\alpha>1$, and $\varphi\in C^{\infty}_{0}(\mathbb{R}^n \times(0,\infty))$.  Then 
 assuming we can interchange the order of integration in the following
 calculation (we will justify this after the calculation) and using
 the fact that $\Phi_\beta*\Phi_\gamma=\Phi_{\alpha}$ (see
 \cite[Lemma 5.1]{T}), we find that
 \begin{align}
  \notag (HJ_\alpha f)\varphi&=J_\alpha f(H^* \varphi)=(\Phi_\alpha *f)(H^* \varphi)\\
  \notag &=(\Phi_\beta*\Phi_\gamma *f)(H^* \varphi)\\
  \notag &=\iint\,\iint\Phi_\gamma (\eta-\xi,\zeta-\tau)f(\xi,\tau)\iint\Phi_\beta (x-\eta,t-\zeta)H^* \varphi(x,t)\,dxdt\,d\xi d\tau\,d\eta d\zeta\\
  \label{N7} &=\iint J_\gamma f(\eta,\zeta)\left(\iint\Phi_\beta (x-\eta,t-\zeta)H^* \varphi(x,t)\,dxdt\right)d\eta d\zeta.
 \end{align}
 Taking $\beta=1$ and $\gamma=\alpha-1$ in \eqref{N7} and using
 \eqref{N4} we find that 
\[
(HJ_\alpha f)\varphi=(J_{\alpha-1}f)(\varphi). 
\] 
Hence \eqref{N2} 
 holds provided we justify the calculation \eqref{N7} by verifying
 \begin{equation}\label{N8}
  \iint(\Phi_\alpha *f)|H^* \varphi|<\infty.
 \end{equation}
 To do this, choose $b>0$ such that suppt $\varphi\subset B_b (0)\times(0,b)$ and repeat the calculation \eqref{N7} with $\beta=\alpha-1$ and $\gamma=1$, and  
 with $H^* \varphi$ replaced with $|H^* \varphi|$ to obtain
 \begin{equation}\label{N9}
  \iint(\Phi_\alpha *f)|H^* \varphi|=\int^{b}_{0}\int_{\mathbb{R}^n}J_1 f(\eta,\zeta)\left(\int^{b}_{0}\int_{B_b (0)}\Phi_{\alpha-1}(x-\eta,t-\zeta)|H^* \varphi(x,t)|dxdt\right)d\eta d\zeta.
 \end{equation}
 Choose $p^\prime \in(1,\infty)$ such that
 $0<\frac{1}{p}-\frac{1}{p^\prime}<\frac{2}{n+2}$.  Then by 
\cite[Lemma 7.2]{T}, 
$J_1 f\in L^{p^\prime}(\mathbb{R}^n \times(0,b))$.  Hence \eqref{N8} follows from \eqref{N5}, \eqref{N9}, and H\"{o}lder's inequality.
 
 To prove \eqref{N3}, let $\varphi\in C^{\infty}_{0}(\mathbb{R}^n \times(0,\infty))$.  Then using \eqref{N4} and assuming we can interchange the order of 
 integration we see that
 \begin{align*}
  & H(J_1 f)\varphi=(J_1 f)(H^* \varphi)=(\Phi_1 *f)(H^* \varphi)\\
  & =\iint\left(\iint\Phi_1 (x-\xi,t-\tau)H^* \varphi(x,t)dxdt\right)f(\xi,\tau)d\xi d\tau\\
  & =f(\varphi).
 \end{align*}
 Thus \eqref{N3} holds because interchanging the order of integration is validated by using \eqref{N5} and H\"{o}lder's inequality as in the proof of 
 \eqref{N8}. 
\end{proof}

\begin{lem}[\cite{T}, Lemma 7.4] \label{lem3.5}
 Suppose $x\in\mathbb{R}^n$ and $t,\tau\in(0,\infty)$ satisfy
 \begin{equation}\label{7.8}
  |x|^2 <t \quad\text{and}\quad  \frac{t}{4}<\tau<\frac{3t}{4}.
 \end{equation}
 Then
 $$\int_{|\xi|^2 <\tau}\Phi_1 (x-\xi,t-\tau)\,d\xi\geq C(n)>0$$
 where $\Phi_\alpha$ is defined by \eqref{I.9}.
\end{lem}

\section{Proof Theorems \ref{thm2.1} and \ref{thm2.2}}\label{sec5}
In this section we prove Theorems  \ref{thm2.1} and \ref{thm2.2}.

Theorem \ref{thm2.1} is a consequence of the following Theorems
\ref{thm4.1}--\ref{thm4.4} because Theorem \ref{thm4.1} (\ref{thm4.2},
\ref{thm4.3}, \ref{thm4.4}) guarantees under the assumption on
$p,\alpha$, and $\lambda$ in Theorem \ref{thm2.1} the nonexistence (existence,
nonexistence, existence) of nontrivial nonnegative solutions
$f\in X^p$ of the super $J_\alpha$ problem \eqref{Jsup1}, \eqref{Jsup2}
when $(\lambda,\alpha)\in A(B,C,D\cup E)$.

\begin{thm}\label{thm4.1}
  Suppose $f:\mathbb{R}^n \times\mathbb{R}\to[0,\infty)$ is a
  measurable solution of the super $J_\alpha$ problem \eqref{Jsup1},
  \eqref{Jsup2} where $\lambda\in(0,1)$ and $\alpha\in(0,\infty)$ are
  constants.  Then either
 \begin{equation}\label{T1.1}
  f=0\quad\text{a.e. in }\mathbb{R}^n \times\mathbb{R}
 \end{equation}
 or there exists $a\in[0,\infty)$ such that
 \begin{equation}\label{T1.2}
  f=0\quad\text{a.e. in }\mathbb{R}^n \times(-\infty,a),
 \end{equation}
 \begin{equation}\label{T1.3}
  f(x,t)\geq(M(t-a)^\alpha )^{\frac{\lambda}{1-\lambda}}\quad\text{a.e. in }\mathbb{R}^n \times(a,\infty),
 \end{equation}
 and
 \begin{equation}\label{T1.4}
  (J_\alpha f)(x,t)\geq(M(t-a)^\alpha )^{\frac{1}{1-\lambda}}\quad\text{a.e. in }\mathbb{R}^n \times(a,\infty)
 \end{equation}
 where
 \begin{equation}\label{T1.5}
  M=M(\lambda,\alpha)=\frac{\Gamma\left(\frac{\lambda\alpha}{1-\lambda}+1\right)}{\Gamma\left(\alpha+\frac{\lambda\alpha}{1-\lambda}+1\right)}
 \end{equation}
 where $\Gamma$ is the Gamma function.
\end{thm}

\begin{proof}
 Let
 \begin{equation}\label{T1.6}
  a=\sup\{ t\in\mathbb{R}:\| f\|_{L^\infty (\mathbb{R}^n \times(-\infty,t))}=0\}.
 \end{equation}
 Then \eqref{Jsup2} implies $a\geq0$.  If $a=\infty$ then \eqref{T1.1} holds.  Hence we can assume
 \begin{equation}\label{T1.7}
  a\in[0,\infty).
 \end{equation}
 It follows from \eqref{T1.6} and \eqref{T1.7} that \eqref{T1.2} holds and it remains only to prove \eqref{T1.3} and \eqref{T1.4}.  To do this we first 
 prove
 \begin{equation}\label{T1.8}
  f(x,t)\geq(N_0 (t-a)^\alpha )^{\frac{\lambda}{1-\lambda}}\quad\text{a.e. in }\mathbb{R}^n \times(a,\infty)
 \end{equation}
 for some positive constant $N_0 =N_0 (n,\lambda,\alpha)$.
 Let $T>0$ and $x_0 \in\mathbb{R}^n$ be fixed.  To prove \eqref{T1.8} it suffices to prove
 \begin{equation}\label{T1.9}
  f(x,t)\geq(N_0 (t-a)^\alpha )^{\frac{\lambda}{1-\lambda}}\quad\text{for }(x,t)\in\Omega(x_0 ,a, T),
 \end{equation}
 and some positive constant $N_0 =N_0 (n,\lambda,\alpha)$ where
 $$\Omega(x_0 ,t_0 ,T):=\{(x,t)\in\mathbb{R}^n \times\mathbb{R}:|x-x_0 |^2<(t-t_0 )<T\}$$
 because
 $$\mathbb{R}^n \times(a,\infty)=\bigcup_{\substack{x_0 \in\mathbb{R}^n \\
                                                 T>0}}\Omega(x_0 ,a,T).$$
 Let
 \begin{equation}\label{T1.10}
  t_0 \in(a,T+a)
 \end{equation}
 be fixed.  Then to prove \eqref{T1.9}, and hence \eqref{T1.8}, it suffices to prove
 \begin{equation}\label{T1.11}
  f(x,t)\geq(N_0 (t-t_0 )^\alpha )^{\frac{\lambda}{1-\lambda}}\quad\text{for }(x,t)\in\Omega(x_0 ,t_0 ,T)
 \end{equation}
 and some positive constant $N_0=N_0(n,\lambda,\alpha)$
because then sending $t_0$ to $a$ in \eqref{T1.11} we get \eqref{T1.9}.  Define
 \[
g:\mathbb{R}^n\times\mathbb{R}\to[0,\infty)
\]
by
\begin{equation}\label{T1.12}
  g(x,t)=f(x+x_0 ,t+t_0 ).
 \end{equation}
 Then $f$ satsifies \eqref{T1.11},  and hence \eqref{T1.8}, if and only if $g$ satsifies
 \begin{equation}\label{T1.13}
  g(x,t)\geq(N_0 t^\alpha )^{\frac{\lambda}{1-\lambda}}\quad\text{for }(x,t)\in\Omega(0,0,T).
 \end{equation}
 
 It follows from \eqref{T1.6} and \eqref{T1.10} that $J_\alpha f$ is bounded below by a positive constant on bounded subsets of $\mathbb{R}^n \times(t_0 ,\infty)$, in particular on $\Omega(x_0 ,t_0 ,T)$.  Hence by \eqref{Jsup1} and \eqref{T1.12} we see that $g$ is bounded below by a positive constant on $\Omega(0,0,T)$.  Thus there exists a constant $b_0 >0$ such that
 \begin{equation}\label{T1.14}
  g(x,t)\geq(b_0 T^\alpha )^{\frac{\lambda}{1-\lambda}}\geq(b_0 t^\alpha )^{\frac{\lambda}{1-\lambda}}\quad\text{for }(x,t)\in\Omega(0,0,T).
 \end{equation}
 (Note however that $b_0$ may depend not only on $n,\lambda$, and $\alpha$ but also on $x_0 ,t_0$, and $T$.)
 
 Also for $(x,t)\in\mathbb{R}^n \times\mathbb{R}$ we find from \eqref{T1.12} and \eqref{Jsup1} that
 \begin{align*}
  g(x,t)^{1/ \lambda} & =f(x+x_0 ,t+t_0 )^{1/ \lambda}\geq(J_\alpha f)(x+x_0 ,t+t_0 )\\
  & =\int^{t+t_0}_{-\infty}\int_{\mathbb{R}^n}\Phi_\alpha (x+x_0 -\xi,t+t_0 -\tau)f(\xi,\tau)\,d\xi\,d\tau\\
  & =\int^{t}_{-\infty}\int_{\mathbb{R}^n}\Phi_\alpha (x-\bar{\xi},t-\bar{\tau})f(\bar{\xi}+x_0 ,\bar{\tau}+t_0 )\,d\bar{\xi}\,d\bar{\tau}\\
  & =J_\alpha g(x,t).
 \end{align*}
 Thus
 \begin{equation}\label{T1.15}
  g\geq(J_\alpha g)^\lambda \quad\text{in }\mathbb{R}^n \times\mathbb{R}.
 \end{equation}
 Let $\beta:=\frac{\lambda\alpha}{1-\lambda}$.  Then for $(x,t)\in\Omega(0,0,T)$ we obtain from \eqref{T1.15}, \eqref{T1.14}, and Lemma \ref{lem3.5} that
 \begin{align*}
  & g(x,t)^{1/ \lambda}\geq J_\alpha g(x,t)\geq b^{\beta/ \alpha}_{0}\iint_{\Omega(0,0,T)}\Phi_\alpha (x-\xi,t-\tau)\tau^\beta \,d\xi\,d\tau\\
  & \geq b^{\beta/ \alpha}_{0}\int^{3t/4}_{t/4}\frac{(t-\tau)^{\alpha-1}}{\Gamma(\alpha)}\tau^\beta \left(\int_{|\xi|<\sqrt{\tau}}\Phi_1 (x-\xi,t-\tau)d\xi\right)d\tau\\
  & \geq b^{\beta/ \alpha}_{0}C(n)\int^{3t/4}_{t/4}\frac{(t-\tau)^{\alpha-1}}{\Gamma(\alpha)}\tau^\beta d\tau\\
  & =b^{\beta/ \alpha}_{0}\frac{C(n)}{\Gamma(\alpha)}\left(\int^{3/4}_{1/4}(1-s)^{\alpha-1}s^\beta ds\right)t^{\alpha+\beta=\frac{\alpha}{1-\lambda}}.
 \end{align*} 
 Thus letting
 $$N_0 =N_0 (n,\lambda,\alpha)=\frac{C(n)}{\Gamma(\alpha)}\int^{3/4}_{1/4}(1-s)^{\alpha-1}s^{\frac{\lambda\alpha}{1-\lambda}}ds$$
 and $b_1 =b^{\lambda}_{0}N^{1-\lambda}_{0}$ we have
 \begin{equation}\label{T1.16}
  g(x,t)\geq(b_1t^\alpha )^{\frac{\lambda}{1-\lambda}}\quad\text{for }(x,t)\in\Omega(0,0,t)
 \end{equation}
 where
 $$\frac{b_1}{N_0}=\left(\frac{b_0}{N_0}\right)^\lambda .$$
 
 Iterating the method we used to derive \eqref{T1.16} from \eqref{T1.14} we inductively obtain a sequence $\{b_j \}^{\infty}_{j=0}\subset(0,\infty)$ such
 that for $j=1,2,...$ we have
 \begin{equation}\label{T1.17}
  \frac{b_j}{N_0}=\left(\frac{b_{j-1}}{N_0}\right)^\lambda
 \end{equation}
 and
 \begin{equation}\label{T1.18}
  g(x,t)\geq(b_j t^\alpha)^{\frac{\lambda}{1-\lambda}}\quad\text{for }(x,t)\in\Omega(0,0,T).
 \end{equation}
 Since $\lambda\in(0,1)$ it follows from \eqref{T1.17} that
 $$\lim_{j\to\infty}\frac{b_j}{N_0}=1.$$
 Consequently, sending $j$ to $\infty$ in \eqref{T1.18} we obtain \eqref{T1.13} and hence also \eqref{T1.8}.

 Using \eqref{Jsup1}, \eqref{T1.2}, \eqref{T1.8}, and \eqref{T1.5} and making the change of variables $\bar{t}=t-a,\:\bar{\tau}=\tau-a$ we obtain for $(x,t)\in\mathbb{R}^n \times(a,\infty)$  that
 \begin{align*}
  f(x,t)^{1/ \lambda}&\geq J_\alpha f(x,t)\\
  &\geq N^{\frac{\lambda}{1-\lambda}}_{0}\int^{t}_{a}\frac{(t-\tau)^{\alpha-1}}{\Gamma(\alpha)}(\tau-a)^{\frac{\lambda\alpha}{1-\lambda}}d\tau\\
  &=N^{\frac{\lambda}{1-\lambda}}_{0}\int^{\bar{t}}_{0}\frac{(\bar{t}-\bar{\tau})^{\alpha-1}}{\Gamma(\alpha)}\bar{\tau}^{\frac{\lambda\alpha}{1-\lambda}}d\bar{\tau}\\
  &=N^{\frac{\lambda}{1-\lambda}}_{0}M\bar{t}^{\frac{\alpha}{1-\lambda}}
  =(N^{\lambda}_{0}M^{1-\lambda}(t-a)^\alpha )^{\frac{1}{1-\lambda}}.
 \end{align*}
 Thus for $(x,t)\in\mathbb{R}^n \times(a,\infty)$ we have
 \begin{equation}\label{T1.19}
  f(x,t)\geq(J_\alpha f(x,t))^\lambda \geq(N_1 (t-a)^\alpha )^{\frac{\lambda}{1-\lambda}}
 \end{equation}
 where
 $$\frac{N_1}{M}=\left(\frac{N_0}{M}\right)^\lambda .$$
 
 Iterating the method we used to derive \eqref{T1.19} from \eqref{T1.8} we inductively obtain a sequence $\{N_j \}^{\infty}_{j=0}\subset(0,\infty)$ such 
 that for $j=1,2,...$ we have
 \begin{equation}\label{T1.20}
  \frac{N_j}{M}=\left(\frac{N_{j-1}}{M}\right)^\lambda
 \end{equation}
 and
 \begin{equation}\label{T1.21}
  f(x,t)\geq(J_\alpha f(x,t))^\lambda \geq(N_j (t-a)^\alpha
  )^{\frac{\lambda}{1-\lambda}}
\quad\text{for }(x,t)\in\mathbb{R}^n \times(a,\infty).
 \end{equation}
 Since $\lambda\in(0,1)$ it follows from \eqref{T1.20} that
 $$\lim_{j\to\infty}\frac{N_j}{M}=1.$$
 Consequently, sending $j$ to $\infty$ in \eqref{T1.21} we obtain \eqref{T1.3} and \eqref{T1.4}. 
\end{proof}

\begin{thm}\label{thm4.2}
 Suppose $\alpha,T\in(0,\infty)$ and $p\in[1,\infty]$.  Then there exists $a=a(\alpha)>0$ such that a solution
 $$f\in C^\infty (\mathbb{R}^n \times[0,\infty))\cap X^p$$
 of the super $J_\alpha$ problem
 \begin{align*}
  f\geq J_\alpha f  \quad&\text{in }\mathbb{R}^n \times\mathbb{R}\\
  f=0  \qquad&\text{in }\mathbb{R}^n \times(-\infty,0)
 \end{align*}
 is
 \begin{equation}\label{T2.1}
  f(x,t)=e^{a(t+T)}\Phi_1 (x, t+T)\raisebox{2pt}{$\chi$}_{[0,\infty)}(t).
 \end{equation}
\end{thm}

\begin{proof}
  For all $a>0$ the function $f$ given by \eqref{T2.1} is in
  $C^\infty (\mathbb{R}^n \times[0,\infty))$.  Thus, to prove Theorem
  \ref{thm4.2}, it suffices by Lemma \ref{lem3.2}, to show there exists $a=a(\alpha)>0$
  such that the function
  $g:\mathbb{R}^n \times\mathbb{R}\to[0,\infty)$ defined by
 \begin{equation}\label{T2.2}
  g(x,t)=e^{at}\Phi_1 (x,t)
 \end{equation}
 satisfies \eqref{L2.1} and
 \begin{equation}\label{T2.3}
  g\geq J_\alpha g\quad\text{in }\mathbb{R}^n \times\mathbb{R}.
 \end{equation}
 
 By Remark \ref{rem3.1}, $g$ satisfies \eqref{L2.1} for all $a>0$.  Hence it remains only to show there exists $a=a(\alpha)>0$ such that $g$ satisfies \eqref{T2.3} 
 
 The inequality \eqref{T2.3} holds in $\mathbb{R}^n \times(-\infty,0]$ because $g=0$ there.  On the other hand, for $(x,t)\in\mathbb{R}^n \times(0,\infty)$, it follows from \eqref{T2.2} and Lemma \ref{lem3.1} that
 \begin{align}\label{T2.4}
  \notag J_\alpha g(x,t) & =\int^{t}_{0}\left(\int_{\mathbb{R}^n}\Phi_\alpha (x-\xi,t-\tau)\Phi_1 (\xi,\tau)d\xi\right)e^{a\tau}d\tau\\
  & =\Phi_1 (x,t)\int^{t}_{0}\frac{(t-\tau)^{\alpha-1}}{\Gamma(\alpha)}e^{a\tau}d\tau.
 \end{align}
 However, for $t,a>0$ we have
 \begin{align*}
  e^{-at}\int^{t}_{0} & (t-\tau)^{\alpha-1}e^{a\tau}d\tau\leq\int^{t}_{-\infty}(t-\tau)^{\alpha-1}e^{-a(t-\tau)}d\tau\\
  &
    =\frac{1}{a^\alpha}\int^{\infty}_{0}\zeta^{\alpha-1}e^{-\zeta}d\zeta\to0
\quad\text{as }a\to\infty.
 \end{align*}
 It follows therefore from \eqref{T2.4} and \eqref{T2.2} that there exists $a=a(\alpha)>0$ such that $g$ satisfies \eqref{T2.3} in $\mathbb{R}^n \times(0,\infty)$.
\end{proof}

\begin{thm}\label{thm4.3}
 Suppose
 \begin{equation}\label{T3.1}
  f\in X^p
 \end{equation}
 is a nonnegative solution of the super $J_\alpha$ problem
 \eqref{Jsup1}, \eqref{Jsup2} where $p\in[1,\infty)$ and
 \begin{equation}\label{T3.2}
  (\lambda,\alpha)\in C
 \end{equation}
 where $C$ is the region defined in Section \ref{sec2} and 
 graphed in Figure \ref{fig1}.  Then
 \begin{equation}\label{T3.3}
  f=J_\alpha f=0\quad\text{a.e. in }\mathbb{R}^n \times\mathbb{R}.
 \end{equation}
\end{thm}

\begin{proof}
By \eqref{Jsup1} and \eqref{Jsup2}, $f$ satisfies \eqref{T3.3} a.e. in 
$\mathbb{R}^n \times(-\infty,0]$. Hence it suffices to prove 
\begin{equation}\label{neweq}
f=J_\alpha f=0\quad \text{a.e. in }\mathbb{R}^n \times(0,\infty).
\end{equation}
Let $u=J_m f$ where $m$ is the positive integer satisfying $\alpha\le
m<\alpha+1$. By Lemma \ref{new}, $u$ satisfies \eqref{L3.1} and
\eqref{L3.3} and 
 \begin{align*}
  & H^m u (x,t)=f(x,t)\geq(J_\alpha f(x,t))^\lambda \\
  & =\left(\int^{t}_{0}\int_{\mathbb{R}^n}\frac{(t-\tau)^{(\alpha-m)+(m-1)}}{\Gamma(\alpha)}\Phi_1 (x-\xi,t-\tau)f(\xi,\tau)\,d\xi\,d\tau\right)^\lambda \\
  & \geq\left(\frac{\Gamma(m)}{\Gamma(\alpha)}(t+1)^{\alpha-m}J_m f(x,t)\right)^\lambda =\left(\frac{\Gamma(m)}{\Gamma(\alpha)}(t+1)^{\alpha-m}u(x,t)\right)^\lambda .
 \end{align*}
 Hence \eqref{neweq} follows from \eqref{T3.2} and Lemma \ref{lem3.3}.
\end{proof}

\begin{thm}\label{thm4.4}
 Suppose
 \begin{equation}\label{T4.1}
  \lambda>1,\quad 0<\alpha<\frac{n+2}{2}\left(1-\frac{1}{\lambda}\right),\quad p\in[1,\infty)\quad\text{and }T>0.
 \end{equation}
 Then
 \begin{equation}\label{T4.2}
  \beta:=\frac{n+2}{2}-\frac{\lambda\alpha}{\lambda-1}=\frac{n+2}{2}-\frac{\alpha}{1-\frac{1}{\lambda}}>0
 \end{equation}
 and a solution
 $$f\in C^\infty (\mathbb{R}^n \times[0,\infty))\cap X^p$$
 of the super $J_\alpha$ problem \eqref{Jsup1}, \eqref{Jsup2} is
 \begin{equation}\label{T4.3}
  f(x,t)=A\Phi_\beta (x,t+T)\raisebox{2pt}{$\chi$}_{[0,\infty)}(t)
 \end{equation}
 where
 \begin{equation}\label{T4.4}
  A=A(n,\lambda,\alpha)=\left(\frac{(4\pi)^{(\lambda-1)n/2}\Gamma(\alpha+\beta)^\lambda}{\Gamma(\beta)}\right)^{\frac{1}{\lambda-1}}.
 \end{equation}
\end{thm}

\begin{proof}
 Inequality \eqref{T4.2} follows from \eqref{T4.1} and $f$ given by \eqref{T4.3} is clearly in $C^\infty (\mathbb{R}^n \times[0,\infty))$.  Thus, to prove
 Theorem \ref{thm4.4}, it suffices by Lemma \ref{lem3.2} to show that the function $g:\mathbb{R}^n \times\mathbb{R}\to[0,\infty)$ defined by
 $$g(x,t)=A\Phi_\beta (x,t)$$
 satisfies \eqref{L2.1} and \eqref{L2.2}.
 
 By Remark \ref{rem3.1}, $g$ satisfies \eqref{L2.1}.  Hence it remains
 only to show $g$ satsfies \eqref{L2.2}.  The inequality \eqref{L2.2}
 holds in $\mathbb{R}^n \times(-\infty,0]$ because $g=0$ there.  On
 the other hand, for $(x,t)\in\mathbb{R}^n \times(0,\infty)$, it
 follows from \eqref{T4.2} and Lemma \ref{lem3.1} that
 \begin{align*}
  J_\alpha g(x,t)&=A\int^{t}_{0}\left(\int_{\mathbb{R}^n}\Phi_\alpha (x-\xi,t-\tau)\Phi_\beta (\xi,\tau)d\xi\right)d\tau\\
  &=A\Phi_1 (x,t)\int^{t}_{0}\frac{(t-\tau)^{\alpha-1}\tau^{\beta-1}}{\Gamma(\alpha)\Gamma(\beta)}d\tau\\
  &=A\Phi_1 (x,t)\frac{t^{\alpha+\beta-1}}{\Gamma(\alpha+\beta)}=A\Phi_{\alpha+\beta}(x,t)
 \end{align*}
 and thus
 \begin{align*}
  \frac{(J_\alpha g(x,t))^\lambda}{g(x,t)} & =A^{\lambda-1}\frac
{\Gamma(\beta)
t^{\lambda\alpha+(\lambda-1)(\beta-\frac{n+2}{2})}e^{(1-\lambda)|x|^2
                                             /(4t)}}
{(4\pi)^{(\lambda-1)n/2}\Gamma(\alpha+\beta)^\lambda}\\
  & \leq A^{\lambda-1}\frac{\Gamma(\beta)}{(4\pi)^{(\lambda-1)n/2}\Gamma(\alpha+\beta)^\lambda}=1
 \end{align*}
 by \eqref{T4.1}, \eqref{T4.2}, and \eqref{T4.4}.  The proof of
 Theorem \ref{thm4.4} is now complete.
\end{proof}

Since any solution of the approximate $J_\alpha$ problem is, after
scaling, also a solution of the super and sub $J_\alpha$ problems, it
follows from Theorems \ref{thm2.1} and \ref{thm2.3} that if
$\lambda$, $\alpha$, and $p$ satisfy the conditions in Theorem
\ref{thm2.2} and $(\lambda,\alpha)\in A\cup B\cup C\cup D$ then there
do not exist positive constants $C_1$ and $C_2$ such that the
approximate $J_\alpha$ problem has a nontrivial nonnegative solution
$f\in X^p$. Hence Theorem \ref{thm2.2} follows from the following
theorem.

\begin{thm}\label{thm4.5}
 Suppose the constants $\lambda$ and $\alpha$ satisfy
 \begin{equation}\label{T5.1}
  \lambda>1\quad\text{and}\quad 0<\alpha<\frac{n+2}{2}\left(1-\frac{1}{\lambda}\right).
 \end{equation}
 Define $f:\mathbb{R}^n \times\mathbb{R}\to[0,\infty)$ by
 \begin{equation}\label{T5.2}
  f(x,t)=
  \begin{cases}
   t^{-\frac{\alpha\lambda}{\lambda-1}}w(x/\sqrt{t}) & \text{for }(x,t)\in\mathbb{R}^n \times(0,\infty)\\
   0 & \text{for }(x,t)\in\mathbb{R}^n \times(-\infty,0]
  \end{cases}
 \end{equation}
 where
 \begin{equation}\label{T5.3}
  w(z)=e^{-\frac{\lambda|z|^2}{4}}
(|z|^2 +1)^{-\lambda\left(\frac{n+2}{2}-\frac{\alpha\lambda}{\lambda-1}\right)}\quad\text{for }z\in\mathbb{R}^n .
 \end{equation}
 Then
 \begin{equation}\label{T5.4}
  f\in C^\infty (\mathbb{R}^n \times\mathbb{R}\backslash \{(0,0)\})\cap X^1
 \end{equation}
 and $f$ is a solution of the approximate $J_\alpha$ problem
 \eqref{Jap1}, \eqref{Jap2} for some positive constants $C_1$ and $C_2$ depending only on $n,\lambda$, and $\alpha$.
 
 Moreover if $p\in[1,\infty)$ then $f\in X^p$ if and only if
 \begin{equation}\label{T5.5}
  \alpha<\frac{n+2}{2p}\left(1-\frac{1}{\lambda}\right).
 \end{equation}
\end{thm}

\begin{proof}
 We first prove the last sentence of the theorem.  Let $p\in[1,\infty)$.  then for all $t>0$ we find under the change of variables $x=\sqrt{\tau}z$ that
 \begin{align*}
 \iint_{\mathbb{R}^n \times(-\infty,t)}f(x,\tau)^p dxd\tau&=\int^{t}_{0}\tau^{-\frac{\alpha\lambda p}{\lambda-1}}\left(\int_{\mathbb{R}^n}w\left(\frac{x}{\sqrt{\tau}}\right)^p dx\right)d\tau\\
 &=\int^{t}_{0}\tau^{\frac{n}{2}-\frac{\alpha\lambda p}{\lambda-1}}\left(\int_{\mathbb{R}^n}w(z)^p dz\right)d\tau.
 \end{align*}
 Hence, since $\int_{\mathbb{R}^n}w(z)^p dz<\infty$, we see that $f\in X^p$ if and only if
 $$\frac{n}{2}-\frac{\alpha\lambda p}{\lambda-1}>-1$$
 which is equivalent to \eqref{T5.5}.
 
 It follows from \eqref{T5.1} and the last sentence of the theorem
 that $f\in X^1$.  One easily checks that
 $f\in C^\infty (\mathbb{R}^n \times\mathbb{R}\backslash \{(0,0)\})$.
 Hence $f$ satisfies \eqref{T5.4}.
 
 We now complete the proof of the theorem by proving that $f$ satisfies the inequalities \eqref{Jap1}.  Let $(x,t)\in\mathbb{R}^n \times(0,\infty)$ and let
 $\bar{x}=x/ \sqrt{t}$.  Then
 \begin{equation}\label{T5.6}
  f(x,t)=t^{-\frac{\alpha\lambda}{\lambda-1}}w(\bar{x})
 \end{equation}
 and under the variables $\tau=t\bar{\tau}$ and $\xi=\sqrt{t}\bar{\xi}$ we get
 $$\Phi_\alpha (x-\xi,t-\tau)=t^{\alpha-1-n/2}\Phi_\alpha (\bar{x}-\bar{\xi},1-\bar{\tau}).$$
 Thus
 \begin{align*}
  J_\alpha f(x,t) & =\int^{t}_{0}\int_{\mathbb{R}^n}\Phi_\alpha (x-\xi,t-\tau)\tau^{-\frac{\alpha\lambda}{\lambda-1}}w(\xi/ \sqrt{\tau})\,d\xi\,d\tau\\
  & =t^{\alpha-\frac{\alpha\lambda}{\lambda-1}}\int^{1}_{0}\int_{\mathbb{R}^n}\Phi_\alpha (\bar{x}-\bar{\xi},1-\bar{\tau})\bar{\tau}^{-\frac{\alpha\lambda}{\lambda-1}}w(\bar{\xi}/ \sqrt{\bar{\tau}})\,d\bar{\xi}\,d\bar{\tau}\\
  & =t^{-\frac{\alpha}{\lambda-1}}J_\alpha f(\bar{x},1).
 \end{align*}
 Hence letting
 \begin{equation}\label{T5.7}
  I(x)=\Gamma(\alpha)(4\pi)^{n/2}J_\alpha f(x,1)
 \end{equation}
 we obtain from \eqref{T5.6} that
 $$\frac{J_\alpha f(x,t)}{f(x,t)^{1/ \lambda}}=\frac{J_\alpha f(\bar{x},1)}{w(\bar{x})^{1/ \lambda}}=\frac{I(\bar{x})}{\Gamma(\alpha)(4\pi)^{n/2}w(\bar{x})^{1/ \lambda}}$$
 for $(x,t)\in\mathbb{R}^n \times(0,\infty)$.  Thus, since \eqref{Jap1} clearly holds in $\mathbb{R}^n \times(-\infty,0]$, in order to prove \eqref{Jap1} it  
 suffices to prove
 \begin{equation}\label{T5.8}
  0<C_1 \leq\frac{I(x)}{w(x)^{1/ \lambda}}\leq C_2 \quad\text{for }x\in\mathbb{R}^n
 \end{equation}
 where $C_1$ and $C_2$ depend only on $n,\lambda$, and $\alpha$ and from \eqref{T5.7}
 \begin{equation}\label{T5.9}
  I(x)=\int^{1}_{0}(1-\tau)^{\alpha-1-n/2}\tau^{-\frac{\alpha\lambda}{\lambda-1}}\int_{\mathbb{R}^n}e^{-\frac{|x-\xi|^2}{4(1-\tau)}}w(\xi/ \sqrt{\tau})\,d\xi\,d\tau.
 \end{equation} 

 To do this we will need the identity 
 \begin{equation}\label{T5.10}
  (g_a *g_b )(x)=\left(\frac{\pi ab}{a+b}\right)^{n/2}g_{a+b}(x)\quad\text{for }x\in\mathbb{R}^n \text{ and }a,b>0,
 \end{equation}
 where $g_a :\mathbb{R}^n \to(0,\infty)$ is defined by
 $$g_a (x)=e^{-|x|^2 /a}.$$
 This identity can be proved in a straightforward way using the convolution theorem for the Fourier transform and the well-known transform
 $$\hat{g}_a (y)=(\pi a)^{n/2}e^{-\frac{a|y|^2}{4}}$$
 to show that the left and right sides of \eqref{T5.10} have the same Fourier transform.
 
 We first prove the upper bound in \eqref{T5.8}.  Since \eqref{T5.1} implies
 \begin{equation}\label{T5.11}
  \delta:=\frac{n+2}{2}-\frac{\alpha\lambda}{\lambda-1}>0
 \end{equation}
 it follows from \eqref{T5.3} that
 $$w(z)\leq e^{-\frac{\lambda|z|^2}{4}}\quad\text{for }z\in\mathbb{R}^n .$$
 Hence for $x\in\mathbb{R}^n$ we obtain from \eqref{T5.9} that
 $$I(x)\leq\int^{1}_{0}(1-\tau)^{\alpha-1-n/2}\tau^{-\frac{\alpha\lambda}{\lambda-1}}h_\tau (x)d\tau$$
 where
 $$h_\tau (x)=\int_{\mathbb{R}^n}e^{-\frac{|x-\xi|^2}{4(1-\tau)}}e^{-\frac{\lambda|\xi|^2}{4\tau}}d\xi.$$
 Defining
 \begin{equation}\label{T5.12}
  \sigma:=1-1/\lambda\in(0,1)
 \end{equation}
 by \eqref{T5.1} and using \eqref{T5.10} we find that
 \begin{align*}
  h_\tau (x) & =\left(\frac{4\pi(1-\tau)\tau}{\lambda(1-\sigma\tau)}\right)^{n/2}e^{-\frac{|x|^2}{4(1-\sigma\tau)}}\\
  & \leq(4\pi(1-\tau)\tau)^{n/2}e^{-\frac{|x|^2}{4(1-\sigma\tau)}}\quad\text{for }0<\tau<1\text{ and }x\in\mathbb{R}^n .
 \end{align*}
 Thus
 \begin{equation}\label{T5.13}
  I(x)\leq C\int^{1}_{0}
(1-\tau)^{\alpha-1}\tau^{\delta-1}e^{-\frac{|x|^2}{4(1-\sigma\tau)}}d\tau,
 \end{equation}
 where $\delta$ is defined in \eqref{T5.11}.\\
 \smallskip
  
 \noindent \underline{Case I.} Suppose $|x|\leq1$.  Then by \eqref{T5.13}, \eqref{T5.12}, \eqref{T5.11}, and \eqref{T5.3},
 $$I(x)\leq C\int^{1}_{0}(1-\tau)^{\alpha-1}\tau^{\delta-1}d\tau = C\leq Cw(x)^{1/ \lambda}.$$
 That is the upper bound in \eqref{T5.8} holds.\\
 \smallskip

 \noindent \underline{Case II.} Suppose $|x|>1$.  Then by \eqref{T5.13}, \eqref{T5.12}, and \eqref{T5.11} we have
 \begin{align*}
  e^{|x|^2 /4}I(x) & \leq C\int^{1}_{0}(1-\tau)^{\alpha-1}\tau^{\delta-1}e^{-\left(\frac{1}{1-\sigma\tau}-1\right)\frac{|x|^2}{4}}d\tau\\
  & \leq C\int^{1}_{0}(1-\tau)^{\alpha-1}\tau^{\delta-1}e^{-b\tau}d\tau\quad\text{where }b:=\frac{\sigma|x|^2}{4}>\frac{\sigma}{4}>0\\
&\le C\int_0^b\left(1-\frac{s}{b}\right)^{\alpha-1}\left(\frac{s}{b}\right)^{\delta-1}
e^{-s}\frac{1}{b}ds\quad\text{where }s=b\tau\\
& =Cb^{-(\alpha+\delta-1)}(I_1 (b)+I_2 (b))
 \end{align*}
 where
 \begin{align*}
  I_1 (b): & =\int^{b/2}_{0}(b-s)^{\alpha-1}s^{\delta-1}e^{-s}ds\\
  & \leq Cb^{\alpha-1}\int^{b/2}_{0}s^{\delta-1}e^{-s}ds\leq Cb^{\alpha-1}
  \end{align*}
  and
  \begin{align*}
  I_2 (b): & =\int^{b}_{b/2}(b-s)^{\alpha-1}s^{\delta-1}e^{-s}ds\\
  & \leq e^{-b/2}\int^{b}_{0}(b-s)^{\alpha-1}s^{\delta-1}ds\\
  & =e^{-b/2}b^{\alpha+\delta-1}\int^{1}_{0}
(1-\tau)^{\alpha-1}\tau^{\delta-1}d\tau\quad\text{where }s=b\tau\\
  & =C(e^{-b/2}b^\delta )b^{\alpha-1}\leq Cb^{\alpha-1}.
 \end{align*}
 Hence from \eqref{T5.11} we obtain
 $$e^{|x|^2 /4}I(x)\leq Cb^{-\delta}=C(|x|^2 )^{-\left(\frac{n+2}{2}-\frac{\alpha\lambda}{\lambda-1}\right)}.$$
 Thus the upper bound in \eqref{T5.8} holds when $|x|>1$.
 
 To complete the proof of the theorem we now prove the lower bound in
 \eqref{T5.8}. Since $w$ is a positive continuous function on
 $\mathbb{R}^n$ we find from \eqref{T5.9} and for $x\in\mathbb{R}^n$
 that
 \begin{equation}\label{T5.14}
  I(x)\geq C\int^{1}_{0}(1-\tau)^{\alpha-1-n/2}\tau^{-\frac{\alpha\lambda}{\lambda-1}}\int_{|\xi|<\sqrt{\tau}}e^{-\frac{|x-\xi|^2}{4(1-\tau)}}\,d\xi\,d\tau
 \end{equation}
 and thus from \eqref{T5.11} and for $|x|\le 2$ we have
 \begin{align*}
  I(x) & \geq C\int^{1}_{0}(1-\tau)^{\alpha-1-n/2}\tau^{\delta-1}e^{-\frac{9}{4(1-\tau)}}d\tau\\
  & =C\geq Cw(x)^{1/ \lambda}.
 \end{align*}
 Hence it remains only to prove the lower bound in \eqref{T5.8} when
 \begin{equation}\label{T5.15}
  |x|>2.
 \end{equation}
 
 Since for $x\in\mathbb{R}^n$ and $\tau>0$, the expression
 $$\frac{|B_{|x|}(x)\cap B_{\sqrt{\tau}}(0)|}{|B_{\sqrt{\tau}}(0)|}=:V\left(\frac{|x|}{\sqrt{\tau}}\right)$$
 is an increasing function of $\frac{|x|}{\sqrt{\tau}}$ we have
 $$V\left(\frac{|x|}{\sqrt{\tau}}\right)\geq V(1)=\frac{|B_1 (e)\cap B_1 (0)|}{|B_1 (0)|}>0\quad\text{for }0<\sqrt{\tau}<|x|$$
 where $e:=(1,0,...0)\in\mathbb{R}^n$.  It follows therefore from \eqref{T5.14} and \eqref{T5.15} that
 \begin{align}\label{T5.16}
  \notag I(x)&\geq C\int^{1}_{0}(1-\tau)^{\alpha-1-n/2}\tau^{-\frac{\alpha\lambda}{\lambda-1}}\int_{\xi\in B_{|x|}(x)\cap B_{\sqrt{\tau}}(0)}e^{-\frac{|x|^2}{4(1-\tau)}}\,d\xi\,d\tau\\
  &\geq C\int^{1}_{0}(1-\tau)^{-\mu}\tau^{\delta-1}e^{-\frac{a}{1-\tau}}d\tau
 \end{align}
 where $\delta$ is defined in \eqref{T5.11},
 \begin{equation}\label{T5.17}
  \mu:=\frac{n+2}{2}-\alpha,\quad\text{and}\quad a:=\frac{|x|^2}{4}>1
 \end{equation}
 by \eqref{T5.15}
 
 Next making the change of variables
 $s+a=a/(1-\tau)$
 in \eqref{T5.16} and using \eqref{T5.17} we obtain
 \begin{align*}
  I(x) & \geq C\int^{\infty}_{0}\left(\frac{a}{s+a}\right)^{-\mu}\left(\frac{s}{s+a}\right)^{\delta-1}e^{-(s+a)}\frac{a}{(s+a)^2}ds\\
  & =Ca^{1-\mu}e^{-a}\int^{\infty}_{0}(s+a)^{\mu-\delta-1}s^{\delta-1}e^{-s}ds\\
  & =Ca^{-\delta}e^{-a}\int^{\infty}_{0}\left(1+\frac{s}{a}\right)^{\mu-\delta-1}s^{\delta-1}e^{-s}ds\\
  & \geq Ce^{-\frac{|x|^2}{4}}(|x|^2 )^{-\left(\frac{n+2}{2}-\frac{\alpha\lambda}{\lambda-1}\right)}\\
  & \geq Cw(x)^{1/ \lambda}
 \end{align*}
 because
 $$1<1+\frac{s}{a}<1+s\quad\text{for }s>0$$
 by \eqref{T5.17}.  Thus the lower bound in \eqref{T5.8} holds when $|x|>2$.
\end{proof}


\begin{thebibliography}{HM3}

\bibitem{AABP} B. Abdellaoui, A. Attar, R. Bentifour, I. Peral, On
  fractional p-Laplacian parabolic problem with general data,
  Ann. Mat. Pura Appl. (4) 197 (2018) 329--356.

\bibitem{AE} J. Aguirre, M. Escobedo, A Cauchy problem for
  $u_t-\Delta u=u^p$ with $0<p<1$. Asymptotic behaviour of
  solutions. Ann.~Fac.~Sci.~Toulouse Math. 8 (1986/87) 175--203.

\bibitem{AV} E. Affili, E. Valdinoci, Decay estimates for evolution
  equations with classical and fractional time-derivatives,
  J. Differential Equations 266 (2019) 4027--4060.

\bibitem{AMPP} B. Abdellaoui, M. Medina, I. Peral,
  A. Primo, Optimal results for the fractional heat equation
  involving the Hardy potential, Nonlinear Anal. 140 (2016) 166--207.

\bibitem{A} M. Allen, A nondivergence parabolic problem with a
  fractional time derivative, Differential Integral Equations 31
  (2018) 215--230.

\bibitem{ACV} M. Allen,  L. Caffarelli,  A. Vasseur,  A
  parabolic problem with a fractional time
  derivative, Arch. Ration. Mech. Anal. 221 (2016) 603--630.

\bibitem{ACM} I. Athanasopoulos, L. Caffarelli, E. Milakis,
  On the regularity of the non-dynamic parabolic fractional
  obstacle problem, J. Differential Equations 265 (2018) 2614--2647.

\bibitem{BV} M. Bonforte, J. L. V\'azquez,  A priori estimates
  for fractional nonlinear degenerate diffusion equations on bounded
  domains, Arch. Ration. Mech. Anal. 218 (2015) 317--362.

\bibitem{CVW} H. Chen,  L. V\'eron, Y. Wang, Fractional
  heat equations with subcritical absorption having a measure as
  initial data, Nonlinear Anal. 137 (2016) 306--337.

\bibitem{DS} M. G. Delgadino,  S. Smith,  H\"older estimates for
  fractional parabolic equations with critical divergence free drifts,
  Ann. Inst. H. Poincaré Anal. Non Linéaire 35 (2018) 577--604.

\bibitem{DVV} S. Dipierro, E. Valdinoci, V. Vespri, Decay estimates
  for evolutionary equations with fractional time-diffusion,
  J. Evol. Equ. 19 (2019) 435--462.

\bibitem{EH} M. Escobedo, M. A. Herrero, Boundedness and blow up for a
  semilinear reaction-diffusion system, J. Differential Equations 89
  (1991) 176--202.

\bibitem{F} H. Fujita, On the blowing up of solutions of the Cauchy
  problem for $u_t=\Delta u+u^{1+\alpha}$, J. Fac. Sci. Univ. Tokyo
  Sect.~I, 13 (1966), 109--124.

\bibitem{FKRT} G. Furioli, T. Kawakami, B. Ruf, 
  E. Terraneo, Asymptotic behavior and decay estimates of the
  solutions for a nonlinear parabolic equation with exponential
  nonlinearity, J. Differential Equations 262 (2017) 145--180.

\bibitem{GW} C. G. Gal,  M. Warma,  On some degenerate
  non-local parabolic equation associated with the fractional
  p-Laplacian, Dyn. Partial Differ. Equ. 14 (2017) 47--77.


\bibitem{HW} A. Haraux, F. Weissler, Nonuniqueness for a semilinear
  initial value problem. Indiana Univ. Math. J. 31 (1982)
  167--189.

\bibitem{JS}  M. Jleli,  B. Samet, The decay of mass for a
  nonlinear fractional reaction-diffusion equation, Math. Methods
  Appl. Sci. 38 (2015) 1369--1378.

\bibitem{K}  J. Kadlec,  Solution of the first boundary value problem
  for a generalization of the heat equation in classes of functions
  possessing a fractional derivative with respect to the
  time-variable, (Russian) Czechoslovak Math. J. 16 (91) (1966)
  91--113.

\bibitem{KSVZ} J. Kemppainen, J. Siljander, V. Vergara,
  R. Zacher, Decay estimates for time-fractional and other non-local
  in time subdiffusion equations in $\mathbb{R}^d$. Math. Ann. 366
  (2016) 941--979

\bibitem{M} M. Mirzazadeh, Analytical study of solitons to nonlinear
  time fractional parabolic equations. Nonlinear Dynam. 85 (2016)
  2569--2576.

\bibitem{MP} E. Mitidieri, S. I. Pokhozhaev, A priori estimates and
  the absence of solutions of nonlinear partial differential equations
  and inequalities. (Russian) Tr. Mat. Inst. Steklova 234 (2001),
  1--384; translation in Proc. Steklov Inst. Math. 2001, no. 3(234),
  1--362.


\bibitem{MT} L. Molinet, S. Tayachi, Remarks on the Cauchy problem
  for the one-dimensional quadratic (fractional) heat equation,
  J. Funct. Anal. 269 (2015) 2305--2327.


\bibitem{NS} K. Nystr\"om,   O. Sande, Extension properties and boundary
  estimates for a fractional heat operator, Nonlinear Anal. 140
  (2016) 29--37.

\bibitem{OD}  E. Ozbilge,  A. Demir,  Identification of unknown
  coefficient in time fractional parabolic equation with mixed
  boundary conditions via semigroup approach, Dynam. Systems Appl. 24
  (2015) 341--348.

\bibitem{PV} F. Punzo,  E. Valdinoci, Uniqueness in weighted
  Lebesgue spaces for a class of fractional parabolic and elliptic
  equations, J. Differential Equations 258 (2015) 555--587.


\bibitem{QS} P. Quittner, P. Souplet, Superlinear parabolic
  problems. Blow-up, global existence and steady states, Birkhäuser
  Advanced Texts: Basler Lehrbücher. [Birkhäuser Advanced Texts: Basel
  Textbooks] Birkhäuser Verlag, Basel, 2007.


\bibitem{SK} S. G. Samko,  Hypersingular integrals and their
  applications, Analytical Methods and Special Functions, 5, Taylor \&
  Francis, Ltd., London, 2002.



\bibitem{ST} P. R. Stinga,  J. L. Torrea, Regularity theory and
  extension problem for fractional nonlocal parabolic equations and
  the master equation, SIAM J. Math. Anal. 49 (2017) 3893--3924.

\bibitem{SS} F. Sun, P. Shi,  Global existence and non-existence
  for a higher-order parabolic equation with time-fractional term,
  Nonlinear Anal. 75 (2012) 4145--4155.

\bibitem{Su} S. Sugitani, On nonexistence of global solutions for some
  nonlinear integral equations, Osaka J. Math.  12 (1975) 45--51.

\bibitem{T} S. Taliaferro, Pointwise bounds and blow-up for nonlinear
  fractional parabolic inequalities, J. Math. Pures Appl. 133 (2020)
  287-328.

\bibitem{V} V. Varlamov,  Long-time asymptotics for the
  nonlinear heat equation with a fractional Laplacian in a ball,
  Studia Math. 142 (2000) 71--99.

\bibitem{VV} J. L. V\'azquez, B. Volzone,  Symmetrization for
  linear and nonlinear fractional parabolic equations of porous medium
  type, J. Math. Pures Appl. (9) 101 (2014) 553--582.

\bibitem{VPQR} J. L. V\'azquez,  A. de Pablo, F. Quir\'os, 
  A. Rodr\'iguez, Classical solutions and higher regularity for
  nonlinear fractional diffusion equations, J. Eur. Math. Soc. 19
  (2017) 1949--1975.

\bibitem{VZ} V. Vergara, R. Zacher, Optimal decay estimates
  for time-fractional and other nonlocal subdiffusion equations via
  energy methods, SIAM J. Math. Anal. 47 (2015) 210--239.

\bibitem{ZS} Q-G Zhang,  H-R Sun,  The blow-up and global
  existence of solutions of Cauchy problems for a time fractional
  diffusion equation, Topol. Methods Nonlinear Anal. 46 (2015)
  69--92.

\end{thebibliography}
\end{document}